\documentclass{article}
\usepackage{amssymb}
\usepackage{amsthm,bbm,mathtools}

\newcommand{\aff}{\mathop{\rm aff}}

\newcommand{\ap}[1]{\prescript{a}{}{#1}}

\newcommand{\one}{\mathbbm 1}

\def\reals{\mathbb{R}}

\def\uball{\mathbb{B}}
\def\ereals{\overline{\mathbb{R}}}

\def\rinte{\mathop{\rm rint}}

\def\comp{\raise 1pt \hbox{$\scriptstyle\circ$}}

\def\argmin{\mathop{\rm argmin}\limits}

\def\minimize{\mathop{\rm minimize}\limits}

\def\st{\mathop{\rm subject\ to}}

\def\dom{\mathop{\rm dom}\nolimits}

\def\upto{{\raise 1pt \hbox{$\scriptstyle \,\nearrow\,$}}}
\def\downto{{\raise 1pt \hbox{$\scriptstyle \,\searrow\,$}}}

\def\cl{\mathop{\rm cl}\nolimits}

\def\epi{\mathop{\rm epi}}

\def\tos{\rightrightarrows}
\def\FF{(\F_t)_{t=0}^T}
\def\one{\mathbbm 1}
\def\ovr{\mathop{\rm over}}
\def\uball{\mathbb{B}}

\def\D{{\cal D}}

\def\F{{\cal F}}
\def\G{{\cal G}}

\def\N{{\cal N}}

\newtheorem{theorem}{Theorem}
\newtheorem{lemma}[theorem]{Lemma}
\newtheorem{corollary}[theorem]{Corollary}

\newtheorem{example}{Example}
\theoremstyle{definition}

\newtheorem{assumption}{Assumption}
\newtheorem{remark}{Remark}

\title{Shadow price of information in discrete time stochastic optimization}
\author{Teemu Pennanen\thanks{Department of Mathematics, King's College London,
Strand, London, WC2R 2LS, United Kingdom} \and 
Ari-Pekka Perkki\"o\thanks{Department of Mathematics, Technische Universit\"at Berlin, Building MA,
Str. des 17. Juni 136, 10623 Berlin, Germany. % E-mail: {\it perkkioe@math.tu-berlin.de}.
The author is grateful to the Einstein Foundation for the financial support.}}

\begin{document}
\maketitle

\begin{center}
{\bf Dedicated to R. T. Rockafellar on his 80th Birthday}
\end{center}

\begin{abstract}
The shadow price of information has played a central role in stochastic optimization ever since its introduction by Rockafellar and Wets in the mid-seventies. This article studies the concept in an extended formulation of the problem and gives relaxed sufficient conditions for its existence. We allow for general adapted decision strategies, which enables one to establish the existence of solutions and the absence of a duality gap e.g.\ in various problems of financial mathematics where the usual boundedness assumptions fail. As applications, we calculate conjugates and subdifferentials of integral functionals and conditional expectations of normal integrands. We also give a dual form of the general dynamic programming recursion that characterizes shadow prices of information.
%Our existence proof follows the original functional analytic arguments of Rockafellar and Wets but instead of dynamic programming and induction, we give a direct proof in the space of adapted processes. This allows for significant relaxations of the assumptions made on the objective function.
\end{abstract}

%\noindent\textbf{Keywords.} convex stochastic optimization, shadow price of information
%\newline
%\newline
%\noindent\textbf{AMS subject classification codes.} ?? %46N10, 60G07

\section{Introduction}

Let $(\Omega,\F,P)$ be a probability space with a filtration $\FF$ and consider the multistage stochastic optimization problem
\begin{equation}\label{sp}\tag{SP}
\minimize\quad Eh(x)\quad\ovr\ x\in\N,
\end{equation}
where $\N = \{(x_t)_{t=0}^T\,|\, x_t\in L^0(\Omega,\F_t,P;\reals^{n_t})\}$ denotes the space of decision strategies adapted to the filtration, $h$ is a convex normal integrand on $\reals^n\times\Omega$ and $Eh$ denotes the associated integral functional on $L^0(\Omega,\F,P;\reals^n)$. Here and in what follows, $n=n_0+\cdots+n_T$ and $L^0(\Omega,\F,P;\reals^n)$ denotes the linear space of equivalence classes of $\reals^n$-valued $\F$-measurable functions. As usual, two functions are equivalent if they are equal $P$-almost surely. Throughout, we define the expectation of a measurable function as $+\infty$ unless its positive part is integrable.

Problems of the form \eqref{sp} have been extensively studied since their introduction in the mid 70's; see \cite{roc74b,rw74,rw76}. Despite its simple appearance, problem~\eqref{sp} is a very general format of stochastic optimization. Indeed, various pointwise (almost sure) constraints can be incorporated in the objective by assigning $f$ the value $+\infty$ when the constraints are violated. Several examples can be found in the above references. Applications to financial mathematics are given in \cite{pen11c,pp12,pen14}. Our formulation of problem \eqref{sp} extends it's original formulations by allowing for general filtrations $\FF$ as well as general adapted strategies instead of bounded ones. This somewhat technical extension turns out to be quite convenient e.g.\ in financial applications.

We will use the short hand notation $L^\infty:= L^\infty(\Omega,\F,P;\reals^n)$ and define the function $\phi: L^\infty\to\ereals$ by
\begin{align*}
\phi(z) &= \inf_{x\in \N}Eh(x+z).
\end{align*}
{\em We assume throughout that $\phi(0)$ is finite and that $Eh$ is proper on $L^\infty$}. %This is not a real restriction since, as soon as $\dom Eh\cap\N\ne\emptyset$, a change of variables gives $0\in\dom Eh$.
Clearly $\phi(0)$ is the optimum value of \eqref{sp} while in general, $\phi(z)$ gives the optimum value that can be achieved in combination with an essentially bounded {\em nonadapted} strategy $z$. Note also that $\phi(z)=\phi(0)$ for all $z\in L^\infty\cap\N$.

The space $L^\infty$ is in separating duality with $L^1:=L^1(\Omega,\F,P;\reals^n)$ under the bilinear form
\[
\langle z,v\rangle :=E(z\cdot v).
\]
A $v\in L^1$ is said to be a {\em shadow price of information} for problem \eqref{sp} if it is a subgradient of $\phi$ at the origin, i.e., if
\[
\phi(z)\ge\phi(0)+\langle z,v\rangle\quad\forall z\in L^\infty.
\]
The following result, the proof of which is given in the appendix, shows that the shadow price of information has the same fundamental properties here as in Rockafellar and Wets~\cite{rw76} where the primal solutions were restricted to be essentially bounded. Here and in what follows, $\phi^*$ denotes the {\em conjugate} of $\phi$ defined for each $v\in L^1$ as
\[
\phi^*(v)=\sup_{z\in L^\infty}\{\langle z,v\rangle - \phi(z)\}.
\]
The {\em annihilator} of $\N^\infty$ will be denoted by $\N^\perp:=\{v\in L^1\,|\,\langle z,v\rangle =0\ \forall z\in\N^\infty\}$.

\begin{theorem}\label{thm:spi0}
We have $\phi^*=Eh^*+\delta_{\N^\perp}$. In particular, $v\in L^1$ is a shadow price of information if and only if it solves the dual problem
\[
\minimize\quad Eh^*(v)\quad\ovr\ v\in\N^\perp
\]
and the optimum value equals $-\phi(0)$. In this case, an $x\in\N$ is optimal if and only if $Eh(x)<0$ and it minimizes the function $x\mapsto h(x,\omega)-v(\omega)\cdot x$ almost surely. 
\end{theorem}

The notion of a shadow price of information first appeared in a general single period model in Rockafellar~\cite[Example~6 in Section~10]{roc74} and Rockafellar and Wets~\cite[Section~4]{rw75}. Extension to finite discrete time was given in \cite{rw76}. Continuous-time extensions have been studied in Wets~\cite{wet75}, Back and Pliska~\cite{bp87}, Davis~\cite{dav92} and Davis and Burstein~\cite{db92} under various structural assumptions. The shadow price of information has been found useful in formulating dual problems and deriving optimality condition in general parametric stochastic optimization problems; see e.g.\ \cite{rw78,bp87,bpp}. The shadow price of information is useful also in subdifferential calculus involving conditional expectations; see \cite{rw82} and Section~\ref{ssec:2} below. As a further application, we give a dual formulation of the general dynamic programming recursion from \cite{rw76} and \cite{evs76}; see Section~\ref{ssec:3}.

The main result of this paper gives new generalized sufficient conditions for the existence of a shadow price of information for the discrete time problem~\eqref{sp}. Its proof is obtained by extending the original argument of \cite{rw76} and by relaxing some of the technical assumptions made there. As already noted, we do not require the decision strategies to be essentially bounded. This allows one to establish the existence of solutions and the absence of a duality gap e.g.\ in various problems in financial mathematics; see \cite{pp12,per14b}. We also relax the assumptions made in \cite{rw76} on the normal integrand $h$.

We will denote the {\em adapted projection} of an $x\in L^\infty$ by $\ap x$, that is, $(\ap x)_t=E_tx_t$, where $E_t$ denotes the conditional expectation with respect to $\F_t$. We will also use the notation $x^t:=(x_0,\ldots,x_t)$.

\begin{assumption}\label{dom}
For every $z\in\dom Eh\cap L^\infty$ and every $t=0,\ldots,T$, there exists $\hat z\in\dom Eh\cap L^\infty$ such that $E_t z^t=\hat z^t$.
\end{assumption}

It was assumed in \cite{rw76} (conditions C and D, respectively) that the sets $\dom h(\cdot,\omega)$ are closed, uniformly bounded, and ``nonanticipative'' and that there exists a $\mu\in L^1$ such that $|h(x,\omega)|\le\mu(\omega)$ for all $x\in\dom h(\cdot,\omega)$. The nonanticipativity means the projection mappings $D^t(\omega):=\{x^t\,|\, x\in\dom h(\cdot,\omega)\}$ are $\F_t$-measurable for all $t$. These conditions imply Assumption~\ref{dom}. Indeed, if $z\in\dom Eh\cap L^\infty$, then $z^t\in D^t(\omega)$ almost surely and, by Jensen's inequality, $E_tz^t\in\dom h$ almost surely as well. By the measurable selection theorem (see \cite[Corollary~14.6]{rw98}), there exists a $\hat z\in L^0$ such that $\hat z\in\dom h$ and $\hat z^t=E_tz^t$ almost surely. The uniform boundedness of $\dom h$ implies that $\hat z\in L^\infty$ while the upper bound $\mu$ gives $Eh(\hat z)<\infty$.

%In particular, we do not require strict feasibility, closedness and uniform boundedness of $\dom h(\cdot,\omega)$ nor the existence of an integrable function $\mu$ such that $|h(x,\omega)|\le\mu(\omega)$ whenever $h(x,\omega)<\infty$.

We will also use the following.

\begin{assumption}\label{suf}
There exists $\rho\in\reals$ such that, for every $z\in\aff\dom Eh\cap L^\infty$, there exists $x\in\aff\dom Eh\cap\N^\infty$ such that $\|x-z\|\le \rho\|\ap z-z\|$.
\end{assumption}

Assumption~\ref{suf} holds, in particular, if $\ap z\in\aff\dom Eh$ for all $z\in\aff\dom Eh\cap L^\infty$. In the single-step case where $T=0$, this latter condition coincides with Assumption~\ref{dom}. Assumption~\ref{suf} is also implied by the {\em strict feasibility} assumption made in \cite[Theorem~2]{rw76}. Indeed, strict feasibility implies that $\dom Eh$ contains an open ball so that $\aff\dom Eh=L^\infty$. 

In order to clarify the structure and the logic of its proof, we have split our main result in two statements of independent interest, Theorems~\ref{thm:spi} and \ref{thm:drinte} below. Combining them gives the following extension of \cite[Theorem~2]{rw76}.

\begin{theorem}\label{thm:mt}
Let Assumption~\ref{dom} and \ref{suf} hold, and assume that $Eh$ is strongly continuous at a point of $\N^\infty$ relative to $\aff\dom Eh\cap L^\infty$. Then a shadow price of information exists.
\end{theorem}

A sufficient condition for the relative continuity will be given in Theorem~\ref{thm:rinte} below. It is obtained by extending the argument in the proof of \cite[Theorem~2]{roc71}.

%The main results and their proofs are contained in Section~\ref{sec:spi}. Section~\ref{sec:app} applies the results in calculation of conjugates and subdifferentials. Section~\ref{ssec:1} considers integral functionals on the space of essentially bounded adapted processes. Section~\ref{ssec:2} studies conditional expectations of normal integrands and relaxes some of the assumptions made in \cite{rw82} for the interchange rule of subdifferentiation. Section~\ref{ssec:3} returns to dynamic programming and gives recursive formulations of the dual problem in terms of the conjugate of the cost-to-go function.

%Derivatives of cost-to-go function:

%Wets, Stochastic programs with fixed recourse: the equivalent deterministic program, 1974

%Avriel and Williams, The value ofinformation and stochastic programming, 1970

\section{Existence of a shadow price of information}\label{sec:spi}

Our main results are derived by analyzing the auxiliary value function $\tilde\phi:L^\infty\to\ereals$ defined by
\[
\tilde\phi(z) = \inf_{x\in \N^\infty}Eh(x+z).
\]
Here decision strategies are restricted to be essentially bounded like in \cite{rw76}. Clearly $\tilde\phi\ge\phi$. Our strategy is to establish the existence of a subgradient of $\tilde\phi$ at the origin much like in \cite{rw76}. By the following simple lemma, this will then serve as a shadow price of information for the general problem \eqref{sp}. Following \cite{roc70a}, we denote the {\em biconjugate} of a function $f$ by $\cl f:=f^{**}$.

\begin{lemma}\label{lem:cl}
We have $\cl\tilde\phi=\cl\phi$. If $\partial\tilde\phi(0)$ is nonempty, then $\partial\tilde\phi(0)=\partial\phi(0)$.
\end{lemma}

\begin{proof}
By the interchange rule~\cite[Theorem~14.60]{rw98} again, $\tilde\phi^*=Eh^*+\delta_{\N^\perp}$ (see the proof of Theorem~\ref{thm:spi0}). Thus $\tilde\phi^*=\phi^*$, so $\cl\tilde\phi=\cl\phi$. If $\partial\tilde\phi(0)\ne\emptyset$, then $\tilde\phi(0)=\cl\tilde\phi(0)$ so $\tilde\phi(0)=\phi(0)$ (since we always have $\tilde\phi\ge\phi\ge\cl\phi$), by the first part, so $v\in\partial\tilde\phi(0)$ iff $v\in\partial\phi(0)$.
\end{proof}

The general idea in \cite{rw76} was first to prove the existence of a subgradient for $\tilde\phi$ with respect to the pairing of $L^\infty$ with its Banach dual $(L^\infty)^*$. This was then modified to get a subgradient with respect to the pairing of $L^\infty$ with $L^1\subset(L^\infty)^*$. By \cite{yh52}, any $v\in(L^\infty)^*$ can be expressed as $v=v^a+v^s$ where $v^a\in L^1$ and $v^s\in(L^\infty)^*$ is such that there is a decreasing sequence of sets $A^\nu\in\F$ such that $P(A^\nu)\downto 0$ and 
\[
\langle z,v^s\rangle=0
\]
for any $z\in L^\infty$ that vanishes on $A^\nu$. The representation $v=v^a+v^s$ is known as the {\em Yosida--Hewitt decomposition} of $v$. In order to control the singular component $v^s$, we have introduced Assumption~\ref{dom}.
% In order to control the singular component $v^s$, we will impose the following. Here and in what follows, $E_t$ denotes the conditional expectation with respect to $\F_t$.

Below, the {\em strong topology} will refer to the norm topology of $L^\infty$.

\begin{theorem}\label{thm:spi}
Let Assumption~\ref{dom} hold. If $\tilde\phi$ is proper and strongly closed at the origin, then $\phi$ is closed at the origin and $\phi(0)=(\cl\tilde\phi)(0)$. If $\tilde\phi$ is strongly subdifferentiable at the origin, then $\partial\phi(0)=\partial\tilde\phi(0)\neq\emptyset$.
\end{theorem}

\begin{proof}
By Lemma~\ref{lem:cl}, the first claim holds as soon as $\tilde\phi(0)=\cl\tilde\phi(0)$, while the second holds if $\partial\tilde\phi(0)\ne\emptyset$. Strong closedness of $\tilde\phi$ at the origin means that for every $\epsilon>0$ there is a $v\in(L^\infty)^*$ such that $\tilde\phi(0)\le-\tilde\phi^*(v)+\epsilon$, or equivalently,
\begin{align*}
\tilde\phi(z) &\ge \tilde\phi(0) + \langle z,v\rangle - \epsilon \qquad\forall z\in L^\infty\\
\iff Eh(x+z) &\ge \tilde\phi(0) + \langle z,v\rangle-\epsilon\qquad\forall z\in L^\infty,\ x\in \N^\infty\\
\iff\quad Eh(z) &\ge \tilde\phi(0) + \langle z-x,v\rangle-\epsilon\qquad\forall z\in L^\infty,\ x\in \N^\infty,
%&\iff\quad & Eh(\tilde x+x) &\ge \tilde\phi(0) + \langle \tilde x-x,v\rangle-\epsilon\qquad& &\forall \tilde x\in L^\infty,\ x\in \N^\infty
\end{align*}
which means that $v\perp \N^\infty$ and
\begin{equation}\label{eq:spi1}
Eh(z) \ge \tilde\phi(0) + \langle z,v\rangle-\epsilon\qquad \forall z\in L^\infty.
\end{equation}
Similarly, $\tilde\phi$ is strongly subdifferentiable at the origin iff $v\perp \N^\infty$ and \eqref{eq:spi1} holds with $\epsilon=0$.

We will prove the existence of a $v\perp\N^\infty$ which has $v^s=0$ and satisfies \eqref{eq:spi1} with $\epsilon$ multiplied with $2^{T+1}$. Similarly to the above, this means that $\phi$ is closed (if \eqref{eq:spi1} holds with all $\epsilon>0$) or subdifferentiable (if $\epsilon=0$) at the origin with respect to the weak topology. The existence will be proved recursively by showing that if $v\perp\N^\infty$ satisfies \eqref{eq:spi1} and $v_s^s=0$ for $s>t$ (this holds for $t=T$ as noted above), then there exists a $\tilde v\perp\N^\infty$ which satisfies \eqref{eq:spi1} with $\epsilon$ multiplied by $2$ and $\tilde v_t^s=0$ for $s\ge t$. 

Thus, assume that $v_s^s=0$ for $s>t$ and let $\bar\epsilon>0$ and $\bar x\in\N^\infty$ be such that $\tilde\phi(0)\ge Eh(\bar x)-\epsilon$. Combined with \eqref{eq:spi1} and noting that $\langle \bar x,v\rangle=0$, we get
\[
Eh(z) \ge Eh(\bar x) + \langle z-\bar x,v\rangle -\epsilon-\bar\epsilon\qquad\forall z\in L^\infty.
\]
Let $z\in\dom Eh\cap L^\infty$ and let $\hat z$ be as in Assumption~\ref{dom}. By Theorem~\ref{thm:esd} in the appendix, 
\begin{equation}
Eh(z)\ge Eh(\bar x)+\langle z-\bar x,v^a\rangle-\epsilon-\bar\epsilon,\label{pa}
\end{equation}
and
\begin{equation}
0\ge\langle\hat z-\bar x,v^s\rangle-\epsilon-\bar\epsilon.\label{pa2}
\end{equation}
Since $\hat z^t=E_tz^t$ and $v^s_s=0$ for $s>t$ by assumption, \eqref{pa2} means that
\[
0\ge\sum_{s=0}^t\langle E_tz_s-\bar x_s,v_s^s\rangle - \epsilon-\bar\epsilon.
\]
Each term in the sum can be written as $\langle z_s-\bar x_s,E_t^*v_s^s\rangle$, where $E_t^*$ denotes the adjoint of $E_t:L^\infty(\Omega,\F,P;\reals^{n_t})\to L^\infty(\Omega,\F,P;\reals^{n_t})$. Moreover, since $v\perp\N^\infty$, we have $E_t^*v_t=0$ so, in the last term, $E_t^*v_t^s=-E_t^*v_t^a=-E_tv_t^a$. 
%Since $\langle E_t z_t-\bar x_t,v_t\rangle =0$, the last term in the sum can be written as $\langle E_tz_t-\bar x_t,v_t^s\rangle = -\langle z_t-\bar x_t,E_t v_t^a\rangle$, so
%\[
%0\ge\sum_{s=0}^{t-1}\langle z_s-\bar x_s,E_t^*v_s^s\rangle  - \langle z_t-\bar x_t,E_t v_t^a\rangle - \epsilon-\bar\epsilon,
%\]
%where $E_t^*$ denotes the adjoint of $E_t:L^\infty(\Omega,\F,P;\reals^{n_t})\to L^\infty(\Omega,\F,P;\reals^{n_t})$. 
Thus, combining \eqref{pa2} and \eqref{pa} gives
\[
Eh(z)\ge Eh(\bar x)+\langle z-\bar x,\tilde v\rangle-2\epsilon-2\bar\epsilon,
\]
where 
\[
\tilde v_s=
\begin{cases}
v_s+E_t^*v^s_s & \text{for $s<t$},\\
v^a_s-E_tv^a_s & \text{for $s=t$},\\
v_t & \text{for $s>t$}.
\end{cases}
\]
It is easily checked that we still have $\tilde v\in\N^\perp$ but now $\tilde v^s_s=0$ for every $s\ge t$ as desired. Since $\bar\epsilon>0$ was arbitrary and $\langle\bar x,\tilde v\rangle=0$, we see that $\tilde v$ satisfies \eqref{eq:spi1} with $\epsilon$ multiplied by $2$. This completes the proof since $z\in\dom Eh\cap L^\infty$ was arbitrary.
\end{proof}

The general idea of the above proof is from \cite[Theorem~2]{rw76} where the imposed assumptions guarantee the strong continuity of $\tilde\phi$ at the origin, which in turn guarantees subdifferentiability. The following two results give more general conditions under which the subdifferentiability holds.

\begin{theorem}\label{thm:drinte}
Let Assumption~\ref{suf} hold. If $Eh$ is strongly continuous at a point of $\N^\infty$ relative to $\aff\dom Eh\cap L^\infty$, then $\tilde\phi$ is strongly subdifferentiable at the origin.
\end{theorem}

\begin{proof}
We may assume without loss of generality that there exist $M,\epsilon>0$ such that $Eh(z)\le M$ for all $z\in\aff\dom Eh$ with $\|z\|\le\epsilon$. It is straightforward to check that $\dom\tilde\phi=\N^\infty+\dom Eh$ and $\aff\dom\tilde\phi=\N^\infty+\aff\dom Eh$. Assumption~\ref{suf} implies that if $z\in\aff\dom\tilde\phi$, then $z-x_z\in\aff\dom Eh$ for some $x_z\in\N^\infty$ with $\|z-x_z\|\le \rho\|\ap z-z\|$. Indeed, each $z\in\aff\dom\tilde\phi$ can be expressed as $z=x+w$, where $x\in\N^\infty$ and $w\in\aff\dom Eh$, while Assumption~\ref{suf} gives the existence of a $\tilde x_z\in\aff\dom Eh$ such that $\|\tilde x_z-w\|\le \rho\|\ap w-w\|=\rho\|\ap z-z\|$. Setting $x_z:=\tilde x_z+x$, we have $z-x_z=w-\tilde x_z\in\aff\dom Eh$ and $\|z-x_z\|\le \rho\|\ap z-z\|$ as claimed.

Now, if $z\in\aff\dom\tilde\phi$ is such that $\|z\|\le\epsilon/2\rho$, then $\|z-x_z\|\le\epsilon$, so $\tilde\phi(z)\le Eh(z-x_z)\le M$. Since $\tilde\phi(0)$ is finite by assumption, this implies that $\tilde\phi$ is strongly continuous and thus subdifferentiable on $\aff\dom\tilde\phi$; see \cite[Theorem~11]{roc74}. By the Hahn--Banach theorem, relative subgradients on $\aff\dom\tilde\phi$ can be extended to subgradients on $L^\infty$.
\end{proof}

If $Eh$ is a closed proper and convex with $\aff\dom Eh$ closed, then $ Eh$ is continuous on $\rinte_s\dom Eh$, the {\em relative strong interior} of $\dom Eh$ (recall that the relative interior of a set is defined as its interior with respect to its affine hull). Indeed, $\aff\dom Eh$ is a Banach space whenever it is closed, and then $Eh$ is strongly continuous relative to $\rinte_s\dom Eh$; see e.g.\ \cite[Corollary~8B]{roc74}.

The following result gives sufficient conditions for $\aff\dom Eh$ to be strongly closed and $\rinte_s\dom Eh$ to be nonempty. Its proof, contained in the appendix, is obtained by modifying the proof of \cite[Theorem~2]{roc71} which required that $\aff\dom h=\reals^n$ almost surely. Recall that the set-valued mappings $\omega\mapsto\dom h$ and $\omega\mapsto\aff\dom h$ are measurable; see \cite[Proposition~14.8 and Exercise 14.12]{rw98}.

\begin{theorem}\label{thm:rinte}
Assume that the set
\[
\D=\{x\in L^\infty(\dom h) \mid \exists r>0: \uball (x,r)\cap \aff \dom h\subseteq\dom h\ P\text{-a.e.}\}
\]
is nonempty and contained in $\dom Eh$. Then $Eh:L^\infty\rightarrow\ereals$ is closed proper and convex, $\aff\dom Eh$ is closed and $\rinte_s\dom Eh=\D$. In particular, $Eh$ is strongly continuous throughout $\D$ relative to $\aff\dom Eh\cap L^\infty$.
\end{theorem}

\begin{remark}
Under the assumptions Theorem~\ref{thm:rinte}, $Eh$ is subdifferentiable throughout $\D$. Indeed, the construction of $y$ in the proof shows that $y\in\partial Eh(x)$, since $y\in\partial h(x)$ almost surely.
\end{remark}

\begin{example}
The extension of the integrability condition of \cite[Theorem~2]{roc71} in Theorem~\ref{thm:rinte} is needed, for example, in problems of the form
\begin{alignat*}{2}
&\minimize\quad& &Eh_0(x)\quad\ovr\ x\in\N^\infty\\
&\st\quad& & Ax=b\quad P\text{-a.s.},
\end{alignat*}
where $h_0$ is a convex normal integrand such that $h_0(x,\cdot)\in L^1$ for every $x\in\reals^n$, $A$ is a measurable matrix and $b$ is a measurable vector of appropriate dimensions such that the problem is feasible. Indeed, this fits the general format of \eqref{sp} with
\[
h(x,\omega)=
\begin{cases}
h_0(x,\omega) & \text{if $A(\omega)x=b(\omega)$},\\
+\infty & \text{otherwise},
\end{cases}
\]
so that $\aff\dom h=\dom h$ and $\D=\{x\in L^\infty\,|\, Ax=b\ P\text{-a.s.}\}=\dom Eh$.
\end{example}

\section{Calculating conjugates and subgradients}\label{sec:app}

This section applies the results of the previous sections to calculate subdifferentials and conjugates of certain integral functionals and conditional expectations of normal integrands.

\subsection{Integral functionals on $\N^\infty$}\label{ssec:1}

Let $f$ be a normal integrand and consider the associated integral functional $Ef$ with respect to the pairing $\langle\N^\infty,\N^1\rangle$. We assume throughout this section that $\dom Ef\cap\N^\infty\ne\emptyset$.

If $x\in\N^\infty$ and $v\in L^1(\partial f(x))$, then $Ef(x') \ge Ef(x) + \langle x'-x,v\rangle$ for all $x'\in\N^\infty$, so 
\begin{align}\label{eq:inc}
\ap L^1(\partial f(x))\subseteq \partial Ef(x).
\end{align}
The following theorem gives sufficient conditions for this to hold as an equality. We will use the convention that the subdifferential of a function at a point is nonempty unless the function is finite at the point.

\begin{theorem}\label{thm:op}
Assume that $x^*\in\N^1$ is such that the function $\tilde\phi_{x^*}:L^\infty\to\ereals$,
\[
\tilde\phi_{x^*}(z) := \inf_{x\in\N^\infty}E[f(x+z)-(x+z)\cdot x^*]
\]
is closed at the origin. Then 
\[
(Ef)^*(x^*)=\inf_{v\in\N^\perp}Ef^*(x^*+v).
\]
If $\tilde\phi_{x*}$ is subdifferentiable at the origin, then the infimum is attained. % and one has $x^*\in\partial Ef(x)$ when $x^*\in\ap L^1(\partial f(x))$.
If this holds for every $x^*\in\partial Ef(x)$, then
\begin{align*}
\partial Ef(x)=\ap L^1(\partial f(x)).
\end{align*}
\end{theorem}

\begin{proof}
To prove the conjugate formula, note first that $(Ef)^*(x^*)=-\tilde\phi_{x^*}(0)=-\cl\tilde\phi_{x^*}(0)=\inf_{y}\tilde\phi_{x^*}^*(y)$. By the Fenchel inequality, we always have $(Ef)^*(x^*)\le Ef^*(x^*+v)$ for all $v\in\N^\perp$, so we may assume that $\tilde\phi_{x^*}$ is proper. In this case we have the expression $\tilde\phi_{x^*}^*(y)=Ef^*(x^*+y)+\delta_{\N^\perp}(y)$; see the proof of Lemma~\ref{lem:cl}.

Assume now that $\tilde\phi_{x^*}$ is subdifferentiable at the origin for $x^*\in\partial Ef(x)$. Then the infimum in the expression for $(Ef)^*(x^*)$ is attained and $Ef(x)+(Ef)^*(x^*)=\langle x,x^*\rangle$, so there is a $v\in\N^\perp$ such that $E[f(x)+f^*(x^*+v)]=E[x\cdot(x^*+v)]$, and thus $x^*+v\in\partial f(x)$. Clearly, $x^*=\ap(x^*+v)$. Thus, $\partial Ef(x)\supseteq\ap L^1(\partial f(x))$ while the reverse inclusion is always valid by \eqref{eq:inc}.
\end{proof}

Combining the previous theorem with the results of Section~\ref{sec:spi}, we get global conditions when the subdifferential of $Ef$ coincides with the optional projection of the subdifferential of $Ef$ with respect to the pairing $\langle L^\infty,L^1\rangle$.

\begin{corollary}\label{cor:n}
Let $f$ satisfy Assumptions~\ref{dom} and \ref{suf}. %Assume that there exists $\rho\in\reals$ such that for each $z\in\aff\dom Ef\cap L^\infty$ there exists an $x\in\aff\dom Ef\cap\N^\infty$ such that $\|x-z\|\le \rho\|\ap z-z\|$.
 If $Ef$ is strongly continuous at a point of $\N^\infty$ relative to $\aff\dom f\cap L^\infty$, then
\[
(Ef)^*(x^*)=\inf_{v\in\N^\perp}Ef^*(x^*+v)\quad\forall x^*\in\N^1
\]
where the infimum is attained, and
\[
\partial Ef(x) = \ap L^1(\partial f(x)).
\]
\end{corollary}
% \begin{corollary}\label{cor:n}
% Assume that $\ap x\in\dom Ef$ for all $x\in\dom Ef\cap L^\infty$ and that $Ef$ is strongly continuous at a point relative to $\aff\dom Ef\cap L^\infty$. Then
% \[
% (Ef)^*(x^*)=\inf_{v\in\N^\perp}Ef^*(x^*+v)\quad\forall x^*\in\N^1
% \]
% where the infimum is attained, and
% \[
% \partial Ef(x) = \ap L^1(\partial f(x)).
% \]
% \end{corollary}
\begin{proof}
Let $x^*\in\N^1$. Since $\dom Ef\cap\N^\infty\ne\emptyset$, we have $\tilde\phi_{x^*}(0)<\infty$. If $\tilde\phi_{x^*}(0)=-\infty$, then $\tilde\phi_{x^*}$ is trivially closed at the origin. Assume now that $\tilde\phi_{x^*}(0)>-\infty$. The assumed properties of $f$ imply that Assumptions~\ref{dom} and \ref{suf} are satisfied by $h(x,\omega):=f(x,\omega)-x\cdot x^*(\omega)$ and that $Eh$ is continuous at a point of $\N^\infty$ relative to $\aff\dom fh\cap L^\infty$. By Theorem~\ref{thm:drinte} and Theorem~\ref{thm:spi}, $\tilde\phi_{x^*}$ is subdifferentiable at the origin. If $x^*\in\partial(Ef)(x)$, Fenchel's inequality $\tilde\phi_{x^*}(0)\ge E[f(x)-x\cdot x^*]\ge -(Ef)^*(x^*)$ implies $\tilde\phi_{x^*}(0)>-\infty$. The assumptions of Theorem~\ref{thm:op} are thus satisfied.
\end{proof}

Without the assumptions of Corollary~\ref{cor:n}, inclusion \eqref{eq:inc} may be strict. A simple example is given on page~176 of \cite{rw82}.

\begin{remark}
By Theorem~\ref{thm:rinte}, the continuity assumption in Corollary~\ref{cor:n} holds, in particular, if 
\[
\D=\{x\in L^\infty(\dom f) \mid \exists r>0: \uball (x,r)\cap \aff \dom f\subseteq\dom h\ P\text{-a.e.}\}
\]
is nonempty and contained in $\dom Ef$.
\end{remark}

%\begin{example}
%Let $T=0$, $n=1$, $\F_0=\{\Omega,\emptyset\}$, let $\xi$ be a random variable uniformly distributed in $[0,1]$ and define $f(\cdot,\omega)=\delta_{(-\infty,\xi]}$. Now $Ef=\delta_{\reals_-}$, so $\partial Ef(0)={\reals_+}$ whereas $\ap L^1(\partial f(0))=\{0\}$.
%\end{example}

\subsection{Conditional expectation of a normal integrand}\label{ssec:2}

Results of the previous section allow for a simple proof of the interchange rule for subdifferentiation and conditional expectation of a normal integrand. Commutation of the two operations has been extensively studied ever since the introduction of the notion of a conditional expectation of a normal integrand in Bismut~\cite{bis73}; see Rockafellar and Wets~\cite{rw82}, Truffert~\cite{tru91} and the references there in. The results of the previous section allow us to relax some of the continuity assumption made in earlier works.

Given a sub-sigma-algebra $\G\subseteq\F$, the $\G$-{\em conditional expectation} of a normal integrand $f$ is a $\G$-measurable normal integrand $E^\G f$ such that
\[
(E^\G f)(x(\omega),\omega)= E^\G [f(x(\cdot),\cdot)](\omega)\quad P\text{-a.s.}
\]
for all $x\in L^0(\Omega,\G,P;\reals^n)$ such that either $f(x)^+\in L^1$ or $f(x)^-\in L^1$. If $\dom Ef^*\cap L^1(\F)\ne\emptyset$, then the conditional expectation exists and is unique in the sense that if $\tilde f$ is another function with the above property, then $\tilde f(\cdot,\omega)=(E^\G f)(\cdot,\omega)$ almost surely; see e.g.\ \cite[Theorem~2.1.2]{tru91}. % Bismut~\cite{bis73}, Dynkin and Evstigneev~\cite{de76}, Castaing and Valadier~\cite[Section~VIII.9]{cv77}, Thibault~\cite{thi81}, Truffert~\cite{tru91} or Choirat, Hess and Seri~\cite{chs3}.

The {\em $\G$-conditional expectation} of an $\F$-measurable set-valued mapping $S:\Omega\tos\reals^n$ is a $\G$-measurable closed-valued mapping $E^\G S$ such that 
\[
L^1(\G,E^\G S)=\cl\{E^\G v\,|\, v\in L^1(\F,S)\}.
\]
The conditional expectation is well-defined and unique as soon as $S$ admits at least one integrable selection; see Hiai and Umegaki~\cite[Theorem~5.1]{hu77}.

The general form of ``Jensen's inequality'' in the following lemma is from \cite[Corollary 2.1.2]{tru91}. We give a direct proof for completeness.

\begin{lemma}\label{lem:jensen}
If $f$ is a convex normal integrand such that $\dom Ef\cap L^\infty(\G)\ne\emptyset$ and $\dom Ef^*\cap L^1(\F)\ne\emptyset$, then 
%$Ef(x)<\infty$ and $Ef^*(v)<\infty$ for some $x\in L^\infty(\G)$ and $v\in L^1(\F)$, then
\[
(E^\G f)^*(E^\G v)\le E^\G f^*(v)
\]
almost surely for all $v\in L^1(\F)$ and
\[
\partial [E^\G f](x) \supseteq E^\G\partial f(x)
\]
for every $x\in\dom Ef\cap L^0(\G)$.
\end{lemma}

\begin{proof}
Fenchel's inequality $f^*(v)\ge x\cdot v-f(x)$ and the assumption $\dom Ef\cap L^\infty(\G)\ne\emptyset$ imply that $E^\G f^*(v)$ is well defined for all $v\in L^1(\F)$. To prove the first claim, assume, for contradiction, that there is a $v\in L^1(\F)$ and a set $A\in\G$ with $P(A)>0$ on which the inequality is violated. Passing to a subset of $A$ if necessary, we may assume that $E[\one_AE^\G f^*(v)]<\infty$ and thus,
\[
E[\one_A(E^\G f)^*(E^\G v)]>E[\one_AE^\G f^*(v)] = E[\one_Af^*(v)].
\]
This cannot happen since, by Fenchel's inequality
\begin{align*}
E[\one_Af^*(v)] &\ge \sup_{x\in L^\infty(\G)}E\one_A[x\cdot E^\G v - (E^\G f)(x)]=E[\one_A(E^\G f)^*(E^\G v)],
\end{align*}
where the equality follows by applying the interchange rule in $L^\infty(A,\G,P;\reals^n)$.

Given $v\in L^1(\F,\partial f(x))$, we have
\[
f(x)+f^*(v)=x\cdot v
\]
almost surely. Let $A^\nu=\{\|x\|\le \nu\}$ so that $\one_{A^\nu} x$ is bounded. Since $\dom Ef^*\cap L^1(\F)\ne\emptyset$, Fenchel inequality implies that $\one_{A^\nu}f(x)$ integrable. Taking conditional expectations,
\[
\one_{A^\nu}E^\G f(x)+ \one_{A^\nu}E^\G f^*(v)=\one_{A^\nu}x\cdot E^\G v,
\]
so by the first part,
\[
\one_{A^\nu}(E^\G f)(x)+ \one_{A^\nu}(E^\G f)^*(E^\G v)\le \one_{A^\nu}x\cdot E^\G v,
\]
which means that $E^\G v\in\partial(E^\G f)(x)$ almost surely on $A^\nu$. This finishes the proof since $\nu$ was arbitrary. 
\end{proof}

\begin{remark}\label{rem:jensen}
If in Lemma~\ref{lem:jensen}, $f$ is normal $\G$-integrand, then the inequality can be written in the more familiar form $f^*(E^\G v)\le E^\G f^*(v)$. 
\end{remark}

The following gives conditions for the equalities in Lemma~\ref{lem:jensen} to hold. 

\begin{theorem}\label{thm:int}
Let $f$ be a convex normal integrand such that $\dom Ef\cap L^\infty(\G)\ne\emptyset$ and $\dom Ef^*\cap L^1(\F)\ne\emptyset$. If $x^*\in L^1(\G)$ is such that the function $\tilde\phi: L^\infty\to\ereals$,
\[
\tilde\phi(z)=\inf_{x\in L^\infty(\G)}E[f(x+z)-(x+z)\cdot x^*]
\]
is subdifferentiable at the origin, then there is a $v\in L^1(\F)$ such that $E^\G v=0$ and 
\[
(E^\G f)^*(x^*)=E^\G f^*(x^*+v).
\]
If $x\in\dom Ef\cap L^0(\G)$ and the above holds for every $x^*\in L^1(\G;\partial E^\G f(x))$, then
\[
\partial [E^\G f](x) = E^\G\partial f(x).
\]
\end{theorem}

\begin{proof}
Applying Theorem~\ref{thm:op} with $T=0$ and $\F_0=\G$ gives the existence of a $v\in L^1$ such that $E^\G v=0$ and
\[
(Ef)^*(x^*)=Ef^*(x^*+v).
\]
On the other hand, $Ef=E(E^{\G}f)$ by definition, so $(Ef)^*(x^*)=E(E^\G f)^*(x^*)$, by \cite[Theorem~2]{roc68}. The first claim now follows from the fact that $E^\G f^*(x^*+v)\ge (E^\G f)^*(x^*)$ almost surely, by Lemma~\ref{lem:jensen}.

If $x^*\in L^1(\G;\partial E^\G f(x))$, we have
\[
(E^\G f)(x)+(E^\G f)^*(x^*) = x\cdot x^*\quad P\text{-a.s.}
\]
By the first part, there is a $v\in L^1(\F)$ such that $E^\G v=0$ and
\[
(E^\G f)(x)+E^\G f^*(x^*+v) = x\cdot x^*\quad P\text{-a.s.}
\]
It follows that
\[
E[f(x)+f^*(x^*+v) - x\cdot(x^*+v)] = 0,
\]
which by the Fenchel inequality, implies $x^*+v\in\partial f(x)$ so $\partial [E^\G f](x)\subseteq E^\G\partial f(x)$. Combining this with Lemma~\ref{lem:jensen} completes the proof.
\end{proof}

Sufficient conditions for the subdifferentiability condition are again obtained from Theorems~\ref{thm:drinte} and \ref{thm:rinte}.

\begin{corollary}\label{cor:ce}
Let $f$ be a convex normal integrand such that $\dom Ef^*\cap L^1(\F)\ne\emptyset$, $E^\G x\in\dom Ef$ for all $x\in\dom Ef\cap L^\infty$ and $Ef$ is strongly continuous at a point of $L^\infty(\G)$ relative to $\aff\dom Ef\cap L^\infty$. Then for every $x^*\in L^1(\G)$ there is a $v\in L^1(\F)$ such that $E^\G v=0$ and 
\[
(E^\G f)^*(x^*)=E^\G f^*(x^*+v).
\]
Moreover,
\[
\partial [E^\G f](x) = E^\G\partial f(x).
\]
for every $x\in\dom Ef\cap L^0(\G)$.
\end{corollary}

\begin{proof}
Analogously to Corollary~\ref{cor:n}, the additional conditions guarantee the subdifferentiability condition in Theorem~\ref{thm:int}; see the remars after Assumption~\ref{suf}.
\end{proof}

The above subdifferential formula was obtained in \cite{rw82} while the expression for the conjugate was given in \cite[Corollary~2.2.3]{tru91}. Both assumed the stronger condition that $Ef$ be continuous at a point $x\in L^\infty(\G)$ relative to all of $L^\infty(\F)$. A more abstract condition (not requiring the relative continuity assumed here) for the subdifferential formula is given in the corollary in Section~2.2.2 of \cite{tru91}.

Let $g$ be a convex normal integrand. The $\G$-conditional expectation of the epigraphical mapping $\epi g$ is also an epigraphical mapping of some normal integrand as soon as $\epi g$ has an integrable selection; see \cite[p. 136 and 140]{tru91}. We denote by $^\G g$ the normal integrand whose epigraphical mapping is the $\G$-conditional expectation of the epigraphical mapping of $g$. We get from \cite[Theorem 2.1.2 and Corollary 2.1.1.1]{tru91} that
\begin{equation}\label{eq:depi}
({^\G g})^*= E^\G (g^*)
\end{equation}
whenever there exists $y\in \dom Eg\cap L^1$ and $x\in \dom Eg^*\cap L^0(\G)$. Thus results of this section concerning with $(E^\G f)^*$ can be expressed as well in terms ${^\G (}f^*)$.

\subsection{Dynamic programming}\label{ssec:3}

Consider again problem \eqref{sp} and define extended real-valued functions $h_t,\tilde h_t:\reals^{n_1+\dots+n_t}\times\Omega\rightarrow\ereals$ by the recursion
\begin{equation}\label{dp}
\begin{split}
\tilde h_T&=h,\\
h_t &= E_t\tilde h_t,\\
\tilde h_{t-1}(x^{t-1},\omega)&=\inf_{x_t\in\reals^{n_t}}h_t(x^{t-1},x_t,\omega).
\end{split}
\end{equation}
This far reaching generalization of the classical dynamic programming recursion for control systems was introduced in \cite{rw76} and \cite{evs76}. The following result from \cite{pp12} relaxes the compactness assumptions made in \cite{rw76} and \cite{evs76}. In the context of financial mathematics, this allows for various extensions of certain fundamental results in financial mathematics; see \cite{pp12} for details.

\begin{theorem}[\cite{pp12}]\label{thm:dp}
Assume that $h\ge m$ for an $m\in L^1$ and that 
\[
\{x\in\N\,|\,h^\infty(x)\le 0\ P\text{-a.s.}\}
\]
is a linear space. The functions $h_t$ are then well-defined normal integrands and we have for every $x\in\N$ that
\begin{equation}\label{ie}
Eh_t(x^t)\ge \phi(0)\quad t=0,\ldots,T.
\end{equation}
Optimal solutions $x\in\N$ exist and they are characterized by the condition
\[
x_t(\omega)\in\argmin_{x_t}h_t(x^{t-1}(\omega),x_t,\omega)\quad P\text{-a.s.}\quad t=0,\ldots,T.
\]
which is equivalent to having equalities in \eqref{ie}.% Moreover, there is an optimal solution $x\in\N$ such that $x_t\perp N_t$ for every $t=0,\ldots,T$.
%For any $x\in\N$ and $t$ there is an $\tilde x\in\N$ such that $\tilde x^t=x^t$ and
%\[
%h_t(x^t(\omega),\omega) = E_th_s(\tilde x^s(\omega),\omega)\quad P\text{-a.s.}\quad s=t,\ldots,T.
%\]
\end{theorem}

Consider now the dual problem
\[
\minimize\quad Eh^*(v)\quad\ovr\ v\in\N^\perp
\]
from Theorem~\ref{thm:spi0}. We know that the optimum dual value is at least $-\phi(0)$ and that if the values are equal, the shadow prices of information are exactly the dual solutions. Note also that when the functions $h_t$ and $\tilde h_t$ in the dynamic programming equations are well-defined, their conjugates solve the {\em dual dynamic programming} equations
\begin{equation}\label{ddp}
\begin{split}
\tilde g_T&=h^*,\\
g_t &= {^{\F_t}}\tilde g_t,\\
\tilde g_{t-1}(v^{t-1},\omega)&=g_t(v^{t-1},0,\omega).
\end{split}
\end{equation}
Much like Theorem~\ref{thm:dp} characterizes optimal primal solutions in terms of the dynamic programming equations \eqref{dp}, the following result characterizes optimal dual solutions in terms of the dual recursion \eqref{ddp}.

\begin{theorem}
Assume that the dual problem is proper and that there is a feasible $\bar x\in\N^\infty$ for the primal problem. Then the dual dynamic programming equations are well-defined and we have for every $v\in\N^\perp$ that
\begin{equation}\label{dopt1}
Eg_t(E_tv^t)\ge-\phi(0)\quad t=0,\ldots,T.
\end{equation}
In the absence of a duality gap, optimal dual solutions are characterized by having equalities in \eqref{dopt1} while $x\in\N$ and $v\in\N^\perp$ are primal and dual optimal, respectively, if and only if $Eg(x)<\infty$, $Eg^*(v)<\infty$ and
\[
Eg_t^*(x^t)+Eg_t(E_tv^t)=0\quad t=0,\dots,T,
\]
which is equivalent to having
\begin{align*}
%E_{t+1}v^t&\in\partial\tilde h_t(x^t)\quad P\text{-a.s.}\quad t=0,\dots, T,\\
E_tv^t&\in\partial g_t(x^t)\quad P\text{-a.s.}\quad t=0,\dots,T.
\end{align*}
\end{theorem}

\begin{proof}
Let $\bar v\in\N^\perp$ be feasible for the dual problem. We first show inductively that $E_{t+1} \bar v^t\in\dom E\tilde g_t$ and $\bar x^t \in\dom E\tilde g^*_t$ which implies, in particular, that each $g_{t}={^{\F_t}{\tilde g}}_t$ is well-defined. For $t=T$, this is trivial. Assume that the claim holds for some $t\le T$. Then, for every $v\in\N^\perp$, we have
\begin{align}
\tilde g_{t-1}(E_tv^{t-1})= g_t(E_tv^t)&\le E_t\tilde g_t(E_{t+1}v^t)= E_tg_{t+1}(E_{t+1}v^{t+1}),\label{d}
%g_t^*(x^t)&\le E_tg_{t+1}^*(x^{t+1}).\label{p}
\end{align}
where the inequality follows from the induction hypotheses $\bar x^t \in\dom E\tilde g^*_t$ and Lemma~\ref{lem:jensen}. Thus $E_{t} \bar v^{t-1}\in\dom E\tilde g_{t-1}$. By definition, $\tilde g_{t-1}(v^{t-1})=g_t(v^{t-1},0)$, so $\tilde g^*_{t-1}(x^{t-1},\omega)=\cl\inf_{x_t}g^*_t(x^{t-1},x_t,\omega)$. By \eqref{eq:depi}, $g_t^*= E^{\F_t}(\tilde g_t^*)$. Thus, for every $x\in\N\cap\dom Eg$, we have
\begin{align}
\tilde g_{t-1}^*(x^{t-1})\le g_t^*(x^t)&\le E_tg_{t+1}^*(x^{t+1}).\label{p}
\end{align}
Thus $\bar x^{t-1} \in\dom E\tilde g^*_{t-1}$ which finishes the induction proof.

Let $x\in\dom Eg\cap \N$, and $v\in\dom Eg^*\cap\N^\perp$. Combining \eqref{d} and \eqref{p} with the fact that $g_0^*(x_0)\ge -g_0(0)$ gives
\begin{align}\label{pd}
Eg(v)\ge Eg_t(E_tv^t) \ge Eg_0(0) \ge -Eg_0^*(x_0) \ge -Eg_t^*(x^t) \ge - Eg^*(x) 
\end{align}
for all $t$. In particular, \eqref{dopt1} holds. In the absence of duality gap, \eqref{pd} also imply that optimal dual solutions are characterized by having inequalities in \eqref{dopt1}. Likewise, we get from \eqref{pd} that $x$ and $v$ are primal and dual optimal, respectively, if and only if 
\[
Eg_t^*(x^t)+Eg_t(E_tv^t)=0\quad t=0,\dots,T.
\]
By Fenchel's inequality, $g_t^*(x^t)+g_t(E_tv^t)\ge x^t\cdot (E_t v^t)$, so, by \cite[Lemma 1]{per14b}, $E [x^t\cdot (E_t v^t)]=0$ whenever the left side is integrable. Thus $Eg_t^*(x^t)+Eg_t(E_tv^t)=0$ is equivalent to having $g_t^*(x^t)+g_t(E_tv^t)= x^t\cdot (E_t v^t)$ almost surely, which means that
\[
E_tv^t\in\partial g_t(x^t)
\]
almost surely.
\end{proof}

\section{Appendix}

This appendix contains the proofs of Theorems~\ref{thm:spi0} and \ref{thm:rinte} as well as Theorem~\ref{thm:esd} below which was used in the proof of Theorem~\ref{thm:spi}. Both Theorem~\ref{thm:esd} and \ref{thm:rinte} are simple refinements of well-known results on convex integral functionals, both originally due to Terry Rockafellar.

\begin{theorem}\label{thm:esd}
Let $h$ be a convex normal integrand and $\bar z\in\dom Eh\cap L^\infty$. If $v\in(L^\infty)^*$ and $\epsilon\ge 0$ such that
\begin{equation}\label{app1}
Eh(z)\ge Eh(\bar z) + \langle z-\bar z,v\rangle-\epsilon\quad \forall z\in L^\infty,
\end{equation}
then
\[
Eh(z)\ge Eh(\bar z) + \langle z-\bar z,v^a\rangle-\epsilon\quad \forall z\in L^\infty,
\]
and
\[
0\ge\langle z-\bar z,v^s\rangle-\epsilon\quad\forall z\in\dom Eh.
\]
\end{theorem}

\begin{proof}
Let $z\in\dom Eh\cap L^\infty$ and define $z^\nu:= \one_{A^\nu}\bar z+\one_{\Omega\setminus A^\nu}z$ where $A^\nu$ are the sets in the characterization of the singular component $v^s$. We have $h(z^\nu)\to h(z)$ almost surely and $z^\nu\to z$ both weakly and almost surely. Thus, since $h(z^\nu)\le\max\{h(\bar z),h(z)\}$, Fatou's lemma and~\eqref{app1} give,
\begin{align*}
Eh(z)\ge\limsup Eh(z^\nu)&\ge Eh(\bar z)+\limsup\langle z^\nu-\bar z, v\rangle-\epsilon\\
&=Eh(\bar z)+\langle z-\bar z,v^a\rangle-\epsilon,
\end{align*}
where the equality holds since $z^\nu-\bar z= \one_{\Omega\setminus A^\nu}(z-\bar z)$, so that
\[
\langle z^\nu-\bar z,v\rangle = \langle z^\nu-\bar z,v^a\rangle \to \langle z-\bar z,v^a\rangle.
\]
Now let $z^\nu:= \one_{A^\nu}{z} + \one_{\Omega\setminus A^\nu}\bar z$. We have that $h(z^\nu)\to h(\bar z)$ almost surely and $z^\nu\to\bar z$ both weakly and almost surely. Since $h(z^\nu)\le\max\{h(z),h(\bar z)\}$, Fatou's lemma and \eqref{app1} give,
\begin{align*}
Eh(\bar z)\ge\limsup Eh(z^\nu)&\ge Eh(\bar z)+\limsup\langle z^\nu-\bar z, v\rangle-\epsilon\\
&=Eh(\bar z)+\langle z-\bar z,v^s\rangle-\epsilon,
\end{align*}
where the equality holds since $z^\nu-\bar z=\one_{A^\nu}(z-\bar z)$ so that
\begin{align*}
\langle z^\nu-\bar z, v\rangle = \langle z^\nu-\bar z,v^a\rangle+\langle z^\nu-\bar z,v^s\rangle\to \langle z-\bar z,v^s\rangle
\end{align*}
which completes the proof.
\end{proof}

\noindent
{\em Proof of Theorem~\ref{thm:spi0}}.
Let $D:=\{x\in\N\,|\,\exists z\in L^\infty:\ Eh(x+z)<\infty\}$. By the interchange rule~\cite[Theorem~14.60]{rw98},
\begin{align*}
\phi^*(v)&=\sup_{z\in L^\infty}\{\langle z,v\rangle-\phi(z)\}\\
&=\sup_{x\in\N}\sup_{z\in L^\infty}E[z\cdot v-h(x+z)]\\
&=\sup_{x\in D} E[x\cdot v-h^*(v)].
\end{align*}
Since $\dom Eh\cap L^\infty\ne\emptyset$ implies $\N^\infty\subseteq D$, we have $\phi^*(v)=+\infty$ for $v\notin\N^\perp$. By the Fenchel inequality, $h(x)+h^*(v)\ge x\cdot v$ for all $x,v\in\reals^n$, so \cite[Lemma~1]{per14b} implies $E(x\cdot v)=0$ for every $x\in D$ and $v\in\N^\perp\cap\dom Eh^*$. The second claim follows from the first one by noting that $v\in\partial\phi(0)$ if and only if $-\phi^*(v)=\phi(0)$. Finally, $x\in\N$ and $v\in\N^\perp$ are optimal with $Eh(x)+Eh^*(v)=0$, if and only if $Eh(x)<\infty$, $Eh^*(v)<\infty$ and the above Fenchel inequality holds almost surely as an equality, or equivalently, $v\in\partial h(x)$ almost surely. 
% We have $v\in\partial\phi(0)$ iff
% \begin{align*}
% \inf_{z\in L^\infty}\{\phi(z)-\langle z,v\rangle\}&\ge\phi(0)\\
% \iff\inf_{x\in\N}\inf_{z\in L^\infty}\{Eh(x+z)-z\cdot v\}&\ge\phi(0)\\
% \iff\inf_{x\in\N\cap\dom_1Eh} E[x\cdot v-h^*(v)]&\ge\phi(0)\\
% \implies z\in\N^\perp,\ -Eh^*(v)&\ge\phi(0),
% \end{align*}
% where the last implication holds since $\dom Eh\cap L^\infty\ne\emptyset$ implies $\N^\infty\subseteq\N\cap\dom_1Eh$. Since $h(x)+h^*(v)\ge x\cdot v$ for all $x,v\in\reals^n$, we know from [perp-lamma] that feasible primal and dual solutions satisfy $E(x\cdot v)=0$, that primal objective dominates the negative of the dual objective and that $x\in\N$ and $v\in\N^\perp$ are optimal if $E(x\cdot v)=0$ and $h(x)+h^*(v)=x\cdot v$, or equivalently, $v\in\partial h$ almost surely.
\hfill$\square$
\newline

\noindent
{\em Proof of Theorem~\ref{thm:rinte}}.
Translating, if necessary, we may assume $0\in\D$ so that $L^\infty(\aff\dom h)\subseteq\cup_{\lambda>0}\lambda\D\subseteq\aff\D$. By assumption, $\D\subseteq\dom Eh\cap L^\infty\subseteq L^\infty(\dom h)\subseteq L^\infty(\aff\dom h)$. Thus, $\aff\D=\aff(\dom Eh\cap L^\infty)=\aff L^\infty(\dom h)=L^\infty(\aff\dom h)$ which is a closed set. The above also implies $\rinte_s\D\subseteq\rinte_s\dom Eh\subseteq\rinte_s L^\infty(\dom h)$. Clearly $\rinte_s L^\infty(\dom h)\subseteq\D$ while $\rinte_s\D=\D$. It remains to prove that $Eh$ is closed and proper.

Let $\bar r>0$ be such that $\uball (0,\bar r)\cap \aff \dom h\subseteq \rinte_s\dom h$ almost surely and let $\pi(\omega)$ be the projection from $\reals^d$ to $\aff\dom h(\cdot,\omega)$. There exist $x^i\in\reals^d$, $i=0,\dots,d$ and $r>0$ such that $|x^i|<\bar r$ and $\uball(0,r)$ belongs to the interior of the convex hull of $\{x^i \mid i=0,\dots,d\}$. By \cite[Exercise~14.17]{rw98}, $\pi x$ is measurable for every measurable $x$, so each $\pi x^i$ belongs to $\D$ and thus,
\[
\alpha:=\max_{i=0,\dots,d}h(\pi x^i)
\]
is integrable. %Similarly, any $x\in L^\infty(\aff\dom h)$ with $\|x\|<r$ belongs to $\dom Eh$.
Since $0\in\rinte\dom h$ almost surely, the closed convex-valued mapping 
\[
\Gamma(\omega)=\partial h(0,\omega)\cap\aff\dom h(\cdot,\omega)
\]
is nonempty-valued and measurable. Indeed, the measurability follows from \cite[Proposition~14.11 and Theorem~14.56]{rw98}, and nonemptiness follows from \cite[Theorem~23.4]{roc70a} and the simple fact that $\pi(\partial h)\subseteq \partial h$. By \cite[Corollary 14.6]{rw98}, there exists $y\in L^0(\Gamma)$. By the definition of subdifferential,
\begin{align*}
y(\omega)\cdot x\le h(x,\omega)-h(0,\omega)
\end{align*}
for all $x\in\reals^d$, and, in particular, $h^*(y)\le -h(0)$. Therefore,
\begin{align*}
r|y(\omega)| &=\sup_{x\in\uball(0,r)}\{y(\omega)\cdot x\}\\
&=\sup_{x\in\uball(0,r)}\{y(\omega)\cdot \pi(\omega) x\}\\
&\le \sup_{x\in\uball(0,r)} h(\pi(\omega)x,\omega)-h(0,\omega)\\
&\le \alpha(\omega)-h(0,\omega),
\end{align*}
where the second equality holds since $y(\omega)\in\aff\dom h(\cdot,\omega)$ almost surely. Thus, $y\in L^1$ and $y\in\dom Eh^*$ so, by \cite[Theorem~2]{roc68}, $Eh$ is closed and proper.
\hfill$\square$

\bibliographystyle{plain}
\bibliography{sp}

\begin{thebibliography}{10}

\bibitem{bp87}
K.~Back and S.~R. Pliska.
\newblock The shadow price of information in continuous time decision problems.
\newblock {\em Stochastics}, 22(2):151--186, 1987.

\bibitem{bpp}
S.~Biagini, T.~Pennanen, and A.-P.. Perkki\"o.
\newblock Duality and optimality conditions in stochastic optimization and
  mathematical finance.
\newblock {\em manuscript}, 2015.

\bibitem{bis73}
J.-M. Bismut.
\newblock Int\'egrales convexes et probabilit\'es.
\newblock {\em J. Math. Anal. Appl.}, 42:639--673, 1973.

\bibitem{dav92}
M.~H.~A. Davis.
\newblock Dynamic optimization: a grand unification.
\newblock In {\em Proceedings of the 31st IEEE Conference on Decision and
  Control}, volume~2, pages 2035 -- 2036, 1992.

\bibitem{db92}
M.~H.~A. Davis and G.~Burstein.
\newblock A deterministic approach to stochastic optimal control with
  application to anticipative control.
\newblock {\em Stochastics and Stochastics Reports}, 40(3\&4):203--256, 1992.

\bibitem{evs76}
I.~V. Evstigneev.
\newblock Measurable selection and dynamic programming.
\newblock {\em Math. Oper. Res.}, 1(3):267--272, 1976.

\bibitem{hu77}
F.~Hiai and H.~Umegaki.
\newblock Integrals, conditional expectations, and martingales of multivalued
  functions.
\newblock {\em J. Multivariate Anal.}, 7(1):149--182, 1977.

\bibitem{pen11c}
T.~Pennanen.
\newblock Convex duality in stochastic optimization and mathematical finance.
\newblock {\em Mathematics of Operations Research}, 36(2):340--362, 2011.

\bibitem{pen14}
T.~Pennanen.
\newblock Optimal investment and contingent claim valuation in illiquid
  markets.
\newblock {\em Finance Stoch.}, 18(4):733--754, 2014.

\bibitem{pp12}
T.~Pennanen and A.-P. Perkki\"o.
\newblock Stochastic programs without duality gaps.
\newblock {\em Mathematical Programming}, 136(1):91--110, 2012.

\bibitem{per14b}
A.-P. Perkki\"o.
\newblock Stochastic programs without duality gaps for objectives without a
  lower bound.
\newblock {\em manuscript}, 2014.

\bibitem{roc68}
R.~T. Rockafellar.
\newblock Integrals which are convex functionals.
\newblock {\em Pacific J. Math.}, 24:525--539, 1968.

\bibitem{roc70a}
R.~T. Rockafellar.
\newblock {\em Convex analysis}.
\newblock Princeton Mathematical Series, No. 28. Princeton University Press,
  Princeton, N.J., 1970.

\bibitem{roc71}
R.~T. Rockafellar.
\newblock Integrals which are convex functionals. {II}.
\newblock {\em Pacific J. Math.}, 39:439--469, 1971.

\bibitem{roc74}
R.~T. Rockafellar.
\newblock {\em Conjugate duality and optimization}.
\newblock Society for Industrial and Applied Mathematics, Philadelphia, Pa.,
  1974.

\bibitem{roc74b}
R.~T. Rockafellar.
\newblock On the equivalence of multistage recourse models in stochastic
  optimization.
\newblock pages 314--321. Lecture Notes in Econom. and Math. Systems, Vol. 107,
  1975.

\bibitem{rw74}
R.~T. Rockafellar and R.~J.-B. Wets.
\newblock Continuous versus measurable recourse in {$N$}-stage stochastic
  programming.
\newblock {\em J. Math. Anal. Appl.}, 48:836--859, 1974.

\bibitem{rw75}
R.~T. Rockafellar and R.~J.-B. Wets.
\newblock Stochastic convex programming: {K}uhn-{T}ucker conditions.
\newblock {\em J. Math. Econom.}, 2(3):349--370, 1975.

\bibitem{rw76}
R.~T. Rockafellar and R.~J.-B. Wets.
\newblock Nonanticipativity and {$L^1$}-martingales in stochastic optimization
  problems.
\newblock {\em Math. Programming Stud.}, (6):170--187, 1976.
\newblock Stochastic systems: modeling, identification and optimization, II
  (Proc. Sympos., Univ Kentucky, Lexington, Ky., 1975).

\bibitem{rw78}
R.~T. Rockafellar and R.~J.-B. Wets.
\newblock The optimal recourse problem in discrete time:
  {$L\sp{1}$}-multipliers for inequality constraints.
\newblock {\em SIAM J. Control Optimization}, 16(1):16--36, 1978.

\bibitem{rw82}
R.~T. Rockafellar and R.~J.-B. Wets.
\newblock On the interchange of subdifferentiation and conditional expectations
  for convex functionals.
\newblock {\em Stochastics}, 7(3):173--182, 1982.

\bibitem{rw98}
R.~T. Rockafellar and R.~J.-B. Wets.
\newblock {\em Variational analysis}, volume 317 of {\em Grundlehren der
  Mathematischen Wissenschaften [Fundamental Principles of Mathematical
  Sciences]}.
\newblock Springer-Verlag, Berlin, 1998.

\bibitem{tru91}
A.~Truffert.
\newblock Conditional expectation of integrands and random sets.
\newblock {\em Ann. Oper. Res.}, 30(1-4):117--156, 1991.
\newblock Stochastic programming, Part I (Ann Arbor, MI, 1989).

\bibitem{wet75}
R.~J-B Wets.
\newblock On the relation between stochastic and deterministic optimization.
\newblock In A.~Bensoussan and J.L. Lions, editors, {\em Control Theory,
  Numerical Methods and Computer Systems Modelling}, volume 107 of {\em Lecture
  Notes in Economics and Mathematical Systems}, pages 350--361. Springer, 1975.

\bibitem{yh52}
K.~Yosida and E.~Hewitt.
\newblock Finitely additive measures.
\newblock {\em Trans. Amer. Math. Soc.}, 72:46--66, 1952.

\end{thebibliography}

\end{document}